\documentclass[10pt]{article}
\usepackage{amssymb}
\usepackage{mathrsfs}
\usepackage{amsfonts}
\usepackage{tikz}

\usepackage{amsmath,amsthm,amsfonts,amssymb,amscd}

\def\ex{\mathrm{ex}}

\allowdisplaybreaks[4]
\newtheorem {Problem} {Problem}[section]
\newtheorem {Theorem} [Problem]{Theorem}
\newtheorem {Lemma}[Problem]{Lemma}

\newtheorem {Corollary}[Problem]{Corollary}
\newenvironment {Proof}{\noindent {\bf Proof.}}{\hfill\ensuremath{\square}}
\newcommand*{\QEDB}{\hfill\ensuremath{\square}}

\begin{document}

\title{ The bipartite Tur\'{a}n number  and spectral extremum for linear forests\thanks{Corresponding author: Xiao-Dong Zhang (Email: xiaodong@sjtu.edu.cn)}
}

\author{Ming-Zhu Chen\thanks{Ming-Zhu Chen is partly supported by    the National Natural Science Foundation of China (No. 12101166)  and  Hainan Provincial Natural Science Foundation of China (No. 120RC453). E-mail: mzchen@hainanu.edu.cn},  Ning Wang\thanks{E-mail: wangn0606@163.com} \\
School of Science, Hainan University, Haikou 570228, P. R. China, \\
Long-Tu Yuan\thanks{Long-Tu Yuan is partly supported in part by National Natural Science Foundation of China grant (No. 11901554) and Science and Technology Commission
of Shanghai Municipality (No. 19jc1420100).  Email: ltyuan@math.ecnu.edu.cn. }\\
School of Mathematical Sciences and Shanghai Key Laboratory of PMMP, \\East China Normal University, Shanghai 200240, P. R. China
\and  Xiao-Dong Zhang
\thanks{  Xiao-Dong Zhang is partly supported by the National Natural Science Foundation of China (Nos. 11971311, 12026230).  E-mail: xiaodong@sjtu.edu.cn}
\\School of Mathematical Sciences, MOE-LSC, SHL-MAC\\
Shanghai Jiao Tong University,
Shanghai 200240, P. R. China}

\date{}
\maketitle

\begin{abstract}
The bipartite Tur\'{a}n number of a graph $H$, denoted by $\ex(m,n; H)$, is the maximum number of edges in any bipartite graph $G=(X,Y; E)$ with $|X|=m$ and $|Y|=n$ which  does not contain $H$ as a subgraph.  In this paper, we determined $\ex(m,n; F_{\ell})$ for arbitrary  $\ell$  and appropriately large $n$ with comparing to $m$ and $\ell$, where   $F_\ell$  is a linear forest which consists of $\ell$ vertex disjoint paths. Moreover, the extremal graphs have been characterized. Furthermore, these results are used to obtain the maximum spectral radius of   bipartite graphs  which does not contain $F_{\ell}$ as a subgraph and characterize  all extremal graphs which attain the maximum spectral radius.


\end{abstract}

\section{Introduction}
In this paper, we consider only simple and undirected  graphs.
Let $G$ be an undirected simple graph with vertex set
$V(G)=\{v_1,\dots,v_n\}$ and edge set $E(G)$, where $n$ is called the \emph{order} of $G$.
For $v\in V(G)$,  the \emph{neighborhood} $N_G(v)$ of $v$  is $N_G(v)=\{u: uv\in E(G)\}$ and the \emph{degree} $d_G(v)$ of $v$  is the number of vertices adjacent to $v$ in $G$.
We write  $N(v)$ for $N_G(v)$ and  $d(v)$ for $d_G(v)$  if there is no ambiguity.
For a set of vertices $X$, we use $N_G(X)$ to denote $\bigcup_{v\in X}N_G(v)$ and $N_G[X]$ to denote $\bigcap_{v\in X}N_G(v)$.
 For $X,Y\subseteq V(G)$,  $e(X)$ denotes the number of edges contained in $X$  and $e(X,Y)$ denotes the number of edges of $G$ with one end in $X$ and the other in $Y$.
 For two vertex disjoint graphs $G$ and $H$,  we denote   $G\cup H$   by the \emph{union}  of $G$ and $H$,
and $G\nabla H$ by  the \emph{join} of $G$ and $H$ which is obtained by joining every vertex of $G$ to every vertex of $H$.
Moreover, denote by $kG$   the union of $k$ disjoint copies of $G$ and $\overline{G}$ the complement graph of $G$.  In addition, denote by $M_t$ the  disjoint union of $\lfloor t/2 \rfloor$ disjoint copies of edges and $\lceil t/2 \rceil-\lfloor t/2 \rfloor$ isolated vertex (maybe no isolated vertex).
 Furthermore, denote by $Z^k_{m,n}$ the $(m,n)$-bipartite graph obtained by  identifying a vertex of degree $k$ in a complete bipartite graph $K_{k,n}$ and one vertex of degree $m-k$ in a complete bipartite graph $K_{m-k,1}$. In particular,  $Z^1_{m,n}$  is called the {\it double star}.
 Moreover, denote by  $P_k$  a path of order $k$.
%


\subsection{The bipartite Tur\'{a}n extremal problem}
 For two graphs $G$ and $H$, we say that a graph $G$ is \emph{$H$-free} if  it does
not contain a subgraph isomorphic to $H$, i.e., $G$ contains no copy of $H$.
The {\it Tur\'{a}n number}, denoted by $\ex(n,H)$,  of a graph $H $ is the maximum number of edges in  $H$-free graphs on $n$ vertices. It is a fundamental problem of extremal graph theory to determine the value $\ex(n,H)$ for a graph $H$, and which is called the \emph{Tur\'{a}n problem}.  The Erd\H{o}s-Stone-Simonovits theorem  (see \cite{erdos1966,erdos1946}) stated that
$\ex(n,H)=\left(1-\frac{1}{\chi(H)-1}\right)\binom{n}{2}+o(n^2)$, where $\chi(H)$ is the  chromatic number of $H$. Hence it is a challenge and notorious problem to determine the order of magnitude of $\ex(n,H)$ for a bipartite graph $H$.  For a bipartite graph $H$, the bipartite Tur\'{a}n number, denoted by $\ex(m,n; H)$, of a bipartite graph $H$ is the maximum number of edges in an $H$-free bipartite graph $G=(X,Y;E)$ with $|X|=m$, $|Y|=n$ and $m\le n$.  The problem of determine  the value $\ex(m,n; H)$ is called the \emph{bipartite Tur\'{a}n problem}, or the  \emph{Zarakiewicz problem} (may see \cite{babai2009, furedi1996-1,furedi1996-2, furedi2013}), which is closed interconnected to the Tur\'{a}n problem.  F\"{u}redi proved that $\ex(m,n;K_{s,t})\le (s-t+1)^{1/t}nm^{1-1/t}+tm^{2-2/t}+tn$ for $m\ge s$, $n\ge t$, $s\ge t\ge 1$.  Some recent results and progress on $\ex(m,n;H)$ for bipartite graphs $H$  can be referred to \cite{alon2003,sudakov2020,Yuan2019.1} and references therein.


In 1959, Erd\H{o}s and Gallai in \cite{Erdos1959} proved the following pioneering result, which opens a new subject for bipartite graphs.

\begin{Theorem}\label{Thm1.1}(Erd\H{o}s and Gallai \cite{Erdos1959}).
Let $n\geq k$. Then $\ex(n, P_k)\leq(k-2)n/2$.
\end{Theorem}

In 1984, Gy\'arf\'as, Rousseau and Schelp determined $\ex(m,n;P_k)$ and characterized all bipartite extremal graphs for all values of $m,n$ and $k$ in \cite{Gyarfas-Rousseau-Schelp-1984}. 

	\begin{Theorem}\label{Thm1.2}(Gy\'arf\'as, Rousseau and Schelp \cite{Gyarfas-Rousseau-Schelp-1984}).\label{THM: path in bipartite setting}
	For three positive integers $k,m,n$ with $n\geq m\geq p^\prime+1$ and  $p^\prime=\lfloor k/2 \rfloor-1\ge 0$,  the following holds.

(1) If $k$ is even, then
$$\ex(m,n;P_{k})=\left\{
  \begin{array}{ll}
    p^\prime n, & \hbox{for $p^\prime+1\leq m\leq 2p^\prime$;} \\
     p^\prime(m+n-2p^\prime), & \hbox{for $m\geq 2p^\prime+1$.}
  \end{array}
\right.$$
Moreover, the extremal graph is either $K_{p^\prime,m-p^\prime}\cup K_{p^\prime, n-p^\prime}$ or $K_{p^\prime, n}\cup \overline{K}_{m-p^\prime}$.

(2) If $k=3$, then

$$\ex(m,n;P_3)=m.$$
Moreover, the extremal graph is  $M_m\cup \overline{K}_{n-m}$.

(3) If $k=5$, then
$$\ex(m,n;P_5)=\left\{
  \begin{array}{ll}
    n+m, & \hbox{for $m=n$ and $n$ is even;} \\
    n+m-1, & \hbox{otherwise.}
  \end{array}
\right.$$
 Moreover, if $n=m$ is even, then the unique extremal graph consists of copies of  $C_4$, otherwise the extremal graphs consist of copies of $C_4$ and a star or a double star.

(4) If $k$ is odd with $k\geq 7$, then

$$\ex(m,n;P_k)=\left\{
  \begin{array}{ll}
    (p^\prime+1)^2, & \hbox{for $m=n=p^\prime+1$;} \\
    2(p^\prime+1)^2, & \hbox{for $m=n=2p^\prime+2$;} \\
    p^\prime(m+n-2p^\prime), & \hbox{for $m\geq 2p^\prime+3$ or $n>m=2p^\prime+2$;}\\
      p^\prime n+m-p^\prime, & \hbox{otherwise.} \\
  \end{array}
\right.$$
 Moreover, if $n=m=p^\prime+1$, then the unique extremal graph is $K_{p^\prime+1,p^\prime+1}$;
if $m=n=2p^\prime+2$, then the unique extremal graph is $2 K_{p^\prime+1,p^\prime+1}$;
if $m\geq 2p^\prime+3$ or $n>m=2p^\prime+2$, then the unique extremal graph is $K_{p^\prime,n-p^\prime}\cup K_{p^\prime,m-p^\prime}$;
otherwise,  the unique extremal graph is $Z^{p^\prime}_{m,n}$.
	\end{Theorem}
Note that
$\ex(m,n;P_k)=mn$ for  $n\geq m$ and  $1\leq m\leq\lfloor k/2\rfloor-1 $. Combining with Theorem~\ref{THM: path in bipartite setting}, we have the following corollary.

\begin{Corollary}\label{Cor: path in bipartite setting}
Let $k\geq 2$ and $p^\prime=\lfloor k/2\rfloor-1$.
If $n$ is sufficiently larger than $m$, then $\ex(m,n;P_k)\leq \max\{m, p^\prime(m+n-1)\}$.
\end{Corollary}
	
A linear forest is a forest consisting of vertex disjoint paths.
Many researchers focused on determining the Tur\'{a}n numbers for a linear forest.  For the rest of this paper, let  $F_{\ell}=\bigcup_{i=1}^{\ell} P_{k_i}$ be a linear forest with   ${\ell}\geq 1$ and
$k_1\geq \cdots \geq k_{\ell}\geq2$ and $ p=\sum _{i=1}^{\ell} \lfloor k_i/2 \rfloor-1$. In 2011, Gorgol \cite{gorgol} obtained a lower bound for $\ex(n, kP_3)$ by  way of construction.  Later,
Bushaw and Kettle \cite{Bushwa2011} determined the exact value $\ex(n, kP_{\ell})$ for sufficient large $n$, which is  improved by  Lidick\'{y}, Liu, and Palmer \cite{lidicky2013} who determined the Tur\'{a}n numbers for linear forests $F_{\ell}$ when the order of a  graph is  sufficiently large.

\begin{Theorem}\label{Thm1.3}(Lidick\'{y}, Liu and Palmer \cite{lidicky2013}).
Let $F_{\ell}=\bigcup_{i=1}^{\ell} P_{k_i}$ be a linear forest with   ${\ell}\geq 1$ and
$k_1\geq \cdots \geq k_{\ell}\geq2$ and $ p=\sum _{i=1}^{\ell} \lfloor k_i/2 \rfloor-1$. If  there exists an $1\le i\le \ell$  with $k_i\neq 3$,  then for sufficiently large $n$,
$$\ex(n,F_{\ell})=  \binom{p}{2}+ p(n-p)+c,$$
where $c=1$ if all $k_i$ are odd and $c=0$ otherwise.
Moreover, if $c=1$ then  equality holds if and only if $G=K_{p}\nabla (\overline{K}_{n-p-2}\cup K_2)$.
Otherwise, the equality holds if and only if  $G=K_{p}\nabla \overline{K}_{n-p}$.
\end{Theorem}

Later, Campos and Lopes \cite{campos2015}, and Yuan and Zhang \cite{Yuan2017}, independently,  determined  the exact  Tur\'{a}n number for $ k P_3$.
Very recently,  Yuan and Zhang \cite{Yuan2019.2} determined  the exact Tur\'{a}n number for even liner forest which consists of paths on even number of vertices.

In this paper, motivated by the above results,  we determine the bipartite Tur\'{a}n number $\ex(m,n;F_{\ell})$ for a linear forest $F_{\ell}$ for large $n$ and characterize the bipartite extremal graphs. The first main result of this paper can be stated as follow.

\begin{Theorem}\label{linear forest in bipartite graph}
 Let $F_{\ell}=\bigcup_{i=1}^{\ell} P_{k_i}$ be a linear forest with   ${\ell}\geq 1$ and
$k_1\geq \cdots \geq k_{\ell}\geq2$ and $ p=\sum _{i=1}^{\ell} \lfloor k_i/2 \rfloor-1$.
If $p+1\leq m$ and  $n$ is sufficiently larger with  compare to $p$ and $m$, 
then the following holds.\\
(1). If not all $k_1, k_2, \ldots, k_{\ell}$ are odd, then
$$\ex(m,n;F_{\ell})=\left\{
  \begin{array}{ll}
    pn, & \hbox{for $p+1\leq m\leq2p$;} \\
    p\left(n-\left\lfloor k_\ell/2 \right\rfloor+1\right)
+\left(m-p\right)\left(\left\lfloor k_\ell/2 \right\rfloor-1\right) , & \hbox{for $m\geq 2p+1$}.
  \end{array}
\right.$$
Moreover, the extremal graphs are $K_{p,n}\cup \overline{K}_{m-p}$ for $p\leq m\leq2p-1$,  $K_{p,i}\cup K_{p,n-i}$ with $0\leq i\leq \lfloor k_{\ell}/2\rfloor-1$ for $m=2p$ and $K_{p, n-\lfloor k_{\ell}/2\rfloor+1}\cup K_{m-p,\lfloor k_{\ell}/2\rfloor-1}$  for $m\geq 2p+1$.
\\
(2). If all $k_i$'s are odd and $k_\ell\neq 3,5,7$, then
$$\ex(m,n;F_{\ell})=\left\{
  \begin{array}{ll}
    pn+m-p, & \hbox{for $p+1\leq m\leq 2 p$;} \\
     p\left(n-\left\lfloor k_\ell/2 \right\rfloor+1\right)
+\left(m-p\right)\left(\left\lfloor k_\ell/2 \right\rfloor-1\right), & \hbox{for $m\geq 2 p+1$,}
  \end{array}
\right.$$
Moreover, for $\ell \geq 2$, the extremal graphs are $Z_{m,n}^{p}$  for $p+1\leq m\leq2p$ and $K_{p, n-\lfloor k_{\ell}/2\rfloor+1}\cup K_{m-p,\lfloor k_{\ell}/2\rfloor-1}$  for $m\geq 2p+1$.
\\
(3). If all $k_i$'s are odd and $k_\ell=3$, then
$$\ex(m,n;F_{\ell})=\left\{
  \begin{array}{ll}
    pn+m-p, & \hbox{for $k_1=\cdots=k_\ell=3$;} \\
     pn+1, & \hbox{otherwise,}
  \end{array}
\right.$$
Moreover, the unique extremal graphs are  the graph  obtained  from $K_{p,n}$ by joining $m-p$ independent edges connecting  new $m-p$ isolated vertices to $m-p$ vertices with degree $p$ in $K_{p,n}$ respectively for $k_1=\cdots=k_\ell=3$,  and $Z_{p+1,n}^{p}\cup \overline{K}_{m-p-1}$ otherwise.
\\
(4). If all $k_i$'s are odd and $k_\ell=5$, then
$$
      \ex(m,n;F_\ell)=pn +m-p.
$$
Moreover, the unique extremal graph is $Z_{m,n}^{p}$  for $\ell \geq 2$.
\\
(5). If all $k_i$'s are odd and $k_\ell=7$, then
$$\ex(m,n;F_{\ell})=\left\{
  \begin{array}{ll}
    pn+m-p, & \hbox{for $p+1\leq m\leq 3 p$;} \\
     p\left(n-\left\lfloor k_\ell/2 \right\rfloor+1\right)
+\left(m-p\right)\left(\left\lfloor k_\ell/2 \right\rfloor-1\right), & \hbox{for $m\geq 3 p+1$,}
  \end{array}
\right.$$
Moreover, for $\ell \geq 2$, the extremal graphs are $Z_{m,n}^{p}$  for $p+1\leq m\leq 3p$ and $K_{p, n-\lfloor k_{\ell}/2\rfloor+1}\cup K_{m-p,\lfloor k_{\ell}/2\rfloor-1}$  for $m\geq 3p+1$.
\end{Theorem}
The proof of Theorem ~\ref{linear forest in bipartite graph} will be given in Section 3.
%
\subsection{The spectral Tur\'{a}n extremal problem}

The \emph{adjacency matrix}
$A(G)$ of $G$  is the $n\times n$ matrix $(a_{ij})$, where
$a_{ij}=1$ if $v_i$ is adjacent to $v_j$, and $0$ otherwise. The  \emph{spectral radius} of $G$, denoted by $\lambda(G)$, is the largest eigenvalue of $A(G)$,  while  the least eigenvalue of $A(G)$ is denoted by $\lambda_n(G)$.
The spectral radius of a graph may at times give some valuable information about the structure of graphs.
For example, it is well-known that $\lambda(G)$ is located between the average degree and the maximum degree of $G$, and the
chromatic number of a graph $G$ is at most $\lambda(G) + 1$.

 The spectral version of the Tur\'{a}n problem may be stated as follows: How large is the spectral radius  $\lambda(G)$ of an $H$-free graph $G$ of order $n$?
 Nikiforov \cite{Nikiforov2011} gave an excellent survey for the spectral Tur\'{a}n problem.

 In particular, Nikiforov \cite{Nikiforov2010} in 2010 determined the maximum spectral radius of a $P_k$-free graph of order $n$, which is spectral counterpart of Theorem~\ref{Thm1.1}.

 \begin{Theorem}(Nikiforov \cite{Nikiforov2010})
Let $G $ be a $P_{k}$-free  graph of order $n\geq2^{4k}$, where $k \geq 1$ and $ p'= \lfloor k/2\rfloor-1$. Then the following holds.\\
(1) If $k$ is even, then $\lambda(G) \leq \lambda(K_{p'}\nabla \overline{K}_{n-p'})$ and equality holds if and only if $G= K_{p'}\nabla \overline{K}_{n-p'}$.\\
(2) If  $k$ is odd, then $\lambda(G) \leq \lambda(K_{p'}\nabla (\overline{K}_{n-p'-2}\cup K_2))$ and equality holds if and only if $G= K_{p'}\nabla (\overline{K}_{n-p'-2}\cup K_2)$.
\end{Theorem}
  For linear forest, Chen, Liu, and Zhang  \cite{Chen2019} in 2019 determined the  maximum spectral  radius of  an $F_{\ell}$-free graph of order $n$ and characterized all the extremal graphs,  which is the spectral counterpart of Theorems~\ref{Thm1.3} also.

  If the host graphs are complete bipartite graphs, then the problem of determining the maximum spectral radius of an $H$-free bipartite graphs is called the bipartite spectral Tur\'{a}n problem or the spectral Zarankiewicz problem (see \cite{babai2009}). In 2015, Zhai, Lin, and Gong \cite{Zhai2015} obtained the maximum spectral radius of a $P_k$-free bipartite graph and  characterized all the extremal graphs, which is also the spectral counterpart of Theorem~\ref{Thm1.2}. 

Motivated by the above results, we can employ Theorem~\ref{linear forest in bipartite graph} to determine the   maximum spectral radius of $F_{\ell}$-free bipartite graphs  and  characterize all the extremal graphs. As an application, we determine the  minimum least eigenvalue $\lambda_n(G)$ of an $F_{\ell}$-free bipartite graph and  characterize all the extremal graphs. The second main result of this paper can be stated as follows.

 \begin{Theorem}\label{linear forest}
  Let $F_{\ell}=\bigcup_{i=1}^{\ell} P_{k_i}$ be a linear forest with   ${\ell}\geq 1$,
$k_1\geq \cdots \geq k_{\ell}\geq2$ and $ p=\sum _{i=1}^{\ell} \lfloor k_i/2 \rfloor-1\ge 1$.
If $G$ is an $F_{\ell}$-free bipartite graph of sufficiently large order $n$ with comparing to $p
 $, then
$$\lambda(G)\leq \sqrt{p(n-p)}$$ with equality if and only if $G=K_{p,n-p}$.
\end{Theorem}

\begin{Corollary}\label{Cor}
  Let $F_{\ell}=\bigcup_{i=1}^{\ell} P_{k_i}$ be a linear forest with   ${\ell}\geq 1$,
$k_1\geq \cdots \geq k_{\ell}\geq2$ and $ p=\sum _{i=1}^{\ell} \lfloor k_i/2 \rfloor-1\ge 1$.
 If $G$ is an $F_{\ell}$-free  graph of sufficiently large order $n$ with comparing to $p$, then
$$\lambda_n(G)\geq- \sqrt{p(n-p)}$$ with equality if and only if $G=K_{p,n-p}$.
\end{Corollary}

The rest of this paper is organized as follows. In Section~2, we present some known and necessary results.
In Section~3, we give the proof of Theorems~\ref{linear forest in bipartite graph}.
In Section~4, we  give the proofs of Theorems \ref{linear forest} and Corollary~\ref{Cor}.

\section{Some Lemmas}

\begin{Lemma}\label{lemma for P7 1}
Let $p\geq 3$ and $m_1\leq 2p$. Then $\ex(n_1,m_1;P_7)\leq  pn_1+m_1$ with equality if and only if $n_1=m_1=0$, or $m_1=2p$ and $n_1=2$.
\end{Lemma}
\begin{proof}
If $n_1\leq 2$ or $m_1\leq 2$, then $\ex(n_1,m_1;P_7)=n_1m_1\leq pn_1+m_1$ and equality holds if and only if $n_1=m_1=0$, or $m_1=2p$ and $n_1=2$.
Let $n_1\geq 3$ and $m_1\geq 3$.
If $m_1=n_1=3$, then $\ex(3,3;P_7)= 9<3p+3$.
If $m_1=n_1=6$, then $\ex(6,6;P_7)= 18<6p+6$.
If $\min\{m_1,n_1\}\geq 6$ and $\max\{m_1,n_1\}>6$, then  $\ex(n_1,m_1;P_7)= 2(m_1+n_1-4)=pn_1+m_1  +m_1+(2-p)n_1-8  \leq pn_1+m_1$.
If $\min\{m_1,n_1\}\leq 5$ and $(m_1,n_1)\neq (3,3)$, then $\ex(n_1,m_1;P_7)= m_1+n_1+\min\{m_1,n_1\}-2= pn_1+m_1+(1-p)n_1+  \min\{m_1,n_1\}      -2 <  pn_1+m_1$.
The lemma is proved.
\end{proof}

\begin{Lemma}\label{lemma for P7 2}
Let $p\geq 3$, $m\geq 3p+1$ and $m_1\leq m-p$. Then $\ex(n_1,m_1;P_7)\leq p(n_1-2)+m-p+m_1$ with equality if and only if $n_1=2$ and $m_1=m-p$.
\end{Lemma}
\begin{proof}
If $n_1\leq 2$ or $m_1\leq 2$, then $\ex(n_1,m_1;P_7)=n_1m_1\leq  p(n_1-2)+m-p+m_1$ and the quality holds if and only if $n_1=2$ and $m-p=m_1$.
If $m_1=n_1=3$, then $\ex(3,3;P_7)= 9< m+m_1$.
If $m_1=n_1=6$, then $\ex(6,6;P_7)= 18< 3p+m+m_1$.
If $\min\{m_1,n_1\}\geq 6$ and $\max\{m_1,n_1\}>6$, then  $\ex(n_1,m_1;P_7)=2(m_1+n_1-4)= p(n_1-2)+m-p+m_1 +((2-p)n_1+3p-3m_1-m )< p(n_1-2)+m-p+m_1$.
If $\min\{m_1,n_1\}\leq 5$ and $(m_1,n_1)\neq (3,3)$, then $\ex(n_1,m_1;P_7)= m_1+n_1+\min\{m_1,n_1\}-2=p(n_1-2)+m-p+m_1+((1-p)n_1+3p-m-\min\{m_1,n_1\} -2)  <  p(n_1-2)+m-p+m_1 $.
The lemma is proved.
\end{proof}

\begin{Lemma}\label{large common neighbours}
Let $G=(A,B; E)$ be a  bipartite graph with $|A|=a$ and $ |B|=n$, where $a$ is a constant.  If $a\ge b$ and $e(G)\geq b n$ for sufficiently large $n$ with comparing to $a$, then there exists a vertex set $A^\prime\subseteq A$ with size $t=\lceil b\rceil$ and a constant $c>0$ such that the common neighbours of $A^\prime$ is at least $cn$.
\end{Lemma}

\begin{Proof}
Let $X$ be the set of vertices of $B$ with degree less than $t-1$. Since  $e(G)\geq b n$, we have $(t-1)|X|+(n-|X|)a\geq b n$. Then $|X|\leq \frac{a-b}{a-t+1}n$, which implies that  there are at least $n-|X|\geq \frac{b-t+1}{a-t+1}n $  vertices of $B$ with degree at least $t$. Note that there are ${a \choose t}$ $t$-sets in $A$. By the
pigeonhole principle, there exists a vertex set $A^\prime\subseteq A$ with size $t$ such that the common neighbours of $A$ is at least $cn$, where $c=  \frac{b-t+1}{(a-t+1){a \choose t}}>0.$ The proof is complete.


\end{Proof}

\medskip

We need the following result of  Jackson \cite{jackson1981}.

\begin{Theorem}[Jackson \cite{jackson1981}]\label{cycle}
Let $G=(A,B; E)$ be a bipartite  graph with $|A|=m$ and  $|B|=n$. If    each vertex of $B$ has degree at least $k$ and $m\leq (k-1) \lceil n/(k-1)\rceil$, then $G$ contains a cycle of length  at least $2k$ for sufficiently large $n$.
\end{Theorem}

\begin{Lemma}\label{spec1}
 Let $F_{\ell}=\bigcup_{i=1}^{\ell} P_{k_i}$ be a linear forest with   ${\ell}\geq 1$,
$k_1\geq \cdots \geq k_{\ell}\geq2$ and $ p=\sum _{i=1}^{\ell} \lfloor k_i/2 \rfloor-1\ge 1$.
  Suppose that $G$ has the maximum spectral radius $\lambda$ among all $F_{\ell}$-free connected graphs of  order $n$.  If  $\mathbf x=(x_u)_{u\in V(G)}$ is a positive eigenvector with  maximum entry $1$ corresponding  to  $\lambda$ and $n$ is sufficiently large, then
 $x_u\geq \frac{1}{\lambda }$ for all $u\in V(G)$.
\end{Lemma}

\begin{Proof}
Choose a vertex $w\in V(G)$ with $x_w=1$.
 Since $K_{p,n-p}$ is $F_{\ell}$-free, we have
$$\lambda\geq \lambda(K_{p,n-p})=\sqrt{p(n-p)}.$$
If $u=w$, then $x_u=1\geq \frac{1}{\lambda }$ and we are done. So  we  next suppose that $u\neq w$.
If $u$ is adjacent to $w$,
then by the eigenequation of $A(G)$ on $u$, $$ x_u=\frac{1}{\lambda}\sum_{uv\in E(G)}x_v\geq \frac{x_w}{\lambda}=\frac{1}{\lambda }.$$
If $u$ is not adjacent to $w$, then
let $G_1$ be a graph obtained from $G$ by deleting all edges incident with $u$ and adding an edge $uw$. Note that
$$d_G(w)\geq\sum_{vw\in E(G)}x_v=\lambda x_w=\lambda\geq \sqrt{p(n-p)},$$
which implies that $w$ is a vertex with sufficiently large degree with comparing to $p$.
Then $G_1$ must be $F_{\ell}$-free Since $uw$ is a pendant edge in $G_1$. Hence  \begin{eqnarray*}
  0 &\geq& \lambda(G_1)-\lambda\geq \frac{\mathbf x^{\mathrm{T}}A(G_1)\mathbf x}{\mathbf x^{\mathrm{T}}\mathbf x}- \frac{ {\bf x^T} A(G)\bf x}{\bf {x^T}\bf x}\\
   &=& \frac{2x_u}{\mathbf x^{\mathrm{T}}\mathbf x} \Big(x_w-\sum_{uv\in E(G)}x_v\Big)
 = \frac{2x_u}{\mathbf x^{\mathrm{T}}\mathbf x} \Big(1-\lambda x_u \Big),
\end{eqnarray*}
which implies that $$x_u\geq \frac{1}{\lambda }.$$
This completes the proof.

%
%
%
\end{Proof}

\section{Proof of Theorem~\ref{linear forest in bipartite graph} }
	
\noindent{\bf Proof of Theorem~\ref{linear forest in bipartite graph}.}
We  prove  Theorem~\ref{linear forest in bipartite graph} by the induction on $\ell$. For $\ell=1$,
Theorem~\ref{linear forest in bipartite graph} holds  by Theorem~\ref{THM: path in bipartite setting} directly.
Suppose that Theorem~\ref{linear forest in bipartite graph} holds for  $\ell-1$ with $\ell\geq2$.
Let $G=(A,B; E)$ be an extremal bipartite $F_\ell$-free graph with classes $A$ and $B$, where $|A|=m$ and  $|B|=n$ is sufficiently large  with  comparing to $m$.
Let $F_{\ell-1}=\bigcup_{i=1}^{\ell-1} P_{k_i}$
and $s_j=\sum_{i=1}^{j} \lfloor k_i/2\rfloor$.
Then $s_\ell-1=p$.
We divide the proof into parts (a) and (b).

\medskip

(a). Either not all $k_i$'s are odd, or all $k_i$'s are odd with $k_\ell\notin \{3,5,7\}$. We consider the following three subparts.

\medskip

(a.1).  $p+1\leq m\leq 2p$ and not all $k_i$'s are odd.

Since $G$ is  an extremal bipartite $F_\ell$-free graph and $K_{p,n}\cup \overline{K}_{m-p}$ does not contain a copy of $F_\ell$, we have $e(G)\geq p n$.
By Lemma~\ref{large common neighbours}, there is a set $X_1\subseteq A$ with size $ p$ such that the number of common neighbors of $X_1$ is at least $c_1n$, where $c_1>0$.
Since  $G$ is  $F_\ell$-free and  not all $k_i$'s are odd, there are no vertices of $A\backslash X_1$  whose neighbors belongs to $N(X_1)$.
Since $p+1\leq m\leq 2p$, we have
\begin{eqnarray*}
pn\le e(G) &\leq& p\left|N(X_1)\right| +(m-p)\left(n-\left|N(X_1)\right|\right)\\
   &=& (2p-m)\left|N(X_1)\right|+(m-p)n \\
   &\leq& (2p-m)n+(m-p)n\\
   &=&pn.
\end{eqnarray*}
 Then $e(G)=pn$ and  $G=K_{p,n}\cup \overline{K}_{m-p}$ for $p+1\leq m \leq 2p-1$, and $G=K_{p,i}\cup K_{p,n-i}$ with $0\leq i\leq \lfloor k_{\ell}/2\rfloor-1$ for $m=2p$. So $\ex(m,n; F_{\ell})=pn$ for not all $k_i$'s are odd and  $p+1\leq m\leq 2p$.

(a.2).  $p+1\leq m\leq 2p$ and   all $k_i$'s are odd with $k_\ell\notin \{3,5,7\}$.

 Since  $G$ is  an extremal bipartite $F_\ell$-free graph and $Z_{m,n}^p$ is  $F_\ell$-free,  we have $e(G)\geq p n+m-p$. On the other hand, by Lemma~\ref{large common neighbours}, there is a set $X_2\subseteq A$ with size $p$ such that the number of common neighbors of $X_2$ is at least $c_2n$, where $c_2>0$. Let $U$ be the set of $A\backslash X_2$ such that each vertex in $U$ has at least one neighbor in $N(X_2)$.  Then  it is easy to see that each vertex in $U$ have degree one and all vertices in $U$ share a unique common neighbor  in $N(X_2)$, since   $G$  is  $F_\ell$-free and all $k_i$'s are odd with $k_\ell\notin \{3,5,7\}$.
 Since $p+1\leq m\leq 2p$, we have
\begin{eqnarray*}
pn+m-p\le e(G) &\leq& p\left|N(X_2)\right| +|U|+(m-p-|U|)\left(n-\left|N(X_2)\right|\right)\\
   &=& (2p-m+|U|)\left|N(X_2)\right|+|U|+(m-p-|U|)n \\
   &\leq& (2p-m+|U|)n+|U|+(m-p-|U|)n\\
   &=&pn+|U|\\
   &\leq&pn+m-p.
\end{eqnarray*}
Hence $e(G)=pn+m-p$,  $|U|=m-p$ and $|N(X_2)|=n$,  which implies  $G=Z_{m,n}^p$.
So $\ex(m,n; F_{\ell})=pn+m-p$ for $p+1\leq m\leq 2p$ and   all $k_i$'s are odd with $k_\ell\notin \{3,5,7\}$.

  (a.3).  $m\geq 2p+1$.

    Denote by $f(m, n; a, b)=a(n-b)+(m-a)b$.  Since $K_{p, n-\left\lfloor k_{\ell}/2\right\rfloor+1}\cup K_{m-p,\left\lfloor k_\ell/2 \right\rfloor-1}$ is $F_\ell$-free,
we have
\begin{equation}\label{eq for 1}
e(G)\geq  f(m,n; p, \left\lfloor  k_\ell/2 \right\rfloor-1)= p \left(n-\left\lfloor k_\ell/2 \right\rfloor+1\right)
+(m-p)(\left\lfloor k_\ell/2 \right\rfloor-1).
\end{equation}
Hence for sufficiently large $n $ with comparing to $m$,
\begin{eqnarray*} && e(G)-\ex(m,n; F_{\ell-1})\\
&\ge&   p \left(n-\left\lfloor k_\ell/2 \right\rfloor+1\right)
+(m-p)(\left\lfloor k_\ell/2 \right\rfloor-1)\\
&-&(s_{l-1}-1) \left(n-\left\lfloor k_{\ell-1}/2 \right\rfloor+1\right)-(m-s_{\ell-1}+1)
(\left\lfloor k_{\ell-1}/2 \right\rfloor-1)\\
&=& \lfloor k_{\ell}/2\rfloor(n-2\lfloor k_{\ell}/2\rfloor+2)+(\lfloor k_{\ell-1}/2\rfloor-\lfloor k_{\ell}/2\rfloor)(2s_{l-1}-2-m)\\
&>& 0.
\end{eqnarray*}
Therefore, by the induction hypothesis, $G$ contains a copy of $F_{\ell-1}$.
Let $G^\prime=(A_0,B_0; E_0)=G(A,B)-V(F_{\ell-1})$ with $A_0\subseteq A$ and $B_0\subseteq B$, where $|A_0|=m^\prime$ and $|B_0|=n^\prime$.
Denote by  $a= m-m^\prime\ge 1$  and $b=n-n^\prime\ge 1$. Then $a+b=\sum_{i=1}^{\ell-1} k_{i}$.
In order to  prove that $G=K_{p, n-\left\lfloor k_{\ell}/ 2\right\rfloor+1}\cup K_{m-p,\left\lfloor k_\ell/2 \right\rfloor-1}$, we first prove the following two claims.

\medskip

{\bf Claim~1.} There exists a set $X\subseteq A\backslash A_0$ with size $ s_{\ell-1}$ such that the number of common neighbors of $X$ is at least $cn$ for some constant $c>0$.

\begin{proof}
Since $G$ is $F_\ell$-free, $G^\prime$ is $P_{k_\ell}$-free.
It follows from Corollary~\ref{Cor: path in bipartite setting} that
$$e(A_0,B_0)=e(G')\leq \max\{m',(\lfloor  k_\ell/2\rfloor-1)(n'+m'-1) \}\leq m'+(\lfloor  k_\ell/2\rfloor-1)(n'+m'-1).$$
Since $n$ is sufficiently large  $ m$, combining with (\ref{eq for 1}), we also have
\begin{eqnarray*}
  e(A-A_0,B) &=& e(G)-e(A_0,B_0)-e(A_0,B\backslash B_0) \\
  &\geq&(s_\ell-1) \left(n-\left\lfloor k_\ell/2 \right\rfloor+1\right)
+\left(m-s_\ell+1\right)\left(\left\lfloor k_\ell/2 \right\rfloor-1\right)-\\
&&\left(\left\lfloor k_{\ell}/2\right\rfloor-1\right)\left (n^\prime+m^\prime-1\right)-m'-bm' \\
   &=& s_{\ell -1}n-(b+1)m-(\lfloor k_{\ell}/2 \rfloor-1)(2s_{\ell}-a-b-3)+a(b+1)\\
    &\geq&(s_{\ell -1}-\epsilon)n,
\end{eqnarray*}
for a sufficiently small constant $0<\epsilon<1$.
By Lemma~\ref{large common neighbours}, there exists a set $X\subseteq A\backslash  A_0$ with size $ s_{\ell-1}$ such that the number of common neighbors of $X$ is at least $cn$ for some constant $c>0$. This proves Claim~1.\end{proof}

{\bf Claim~2.}  If $k_{\ell}\ge 4$, then  $|N[X]|\geq n-\left(\left\lfloor k_{\ell}/2 \right\rfloor-1\right)(s_{\ell}+\left\lfloor k_{\ell}/2\rfloor-3 \right)$ for the set  $X$ in Claim 1.

\begin{proof} By Claim 1, the subgraph $[X, B; E_1]$ contains a copy of $F_{\ell-1}$. Hence for sufficiently large $n$ with comparing to $m$, $G-X$ is
$P_{k_\ell}$-free. It follows from Corollary~\ref{Cor: path in bipartite setting} again that
\begin{equation}\label{E2'}
 e(G-X)\leq(\lfloor  k_\ell/2\rfloor-1)(n+ m-s_{\ell-1}- 1).
\end{equation}
By (\ref{eq for 1}) and (\ref{E2'}), we have
\begin{equation}\label{E5}
  \begin{array}{lll}
   e(X,B)&=& e(G)-e(G-X) \\  
 &   \geq& \left(n-\left\lfloor k_\ell/2 \right\rfloor+1\right)(s_\ell-1)
+\left(m-s_\ell+1\right)\left(\left\lfloor k_\ell/2 \right\rfloor-1\right)- \\
&& \left(\left\lfloor k_{\ell}/2 \right\rfloor-1\right) \left(n+m-s_{\ell-1}-1\right)\\
  &=& s_{\ell-1}n-\left(\left\lfloor k_{\ell}/2 \right\rfloor-1\right)(s_{\ell}+\left\lfloor k_{\ell}/2\rfloor-3 \right).
  \end{array}
\end{equation}
 By $|X|=s_{\ell-1}$ and (\ref{E5}), we have
 \begin{eqnarray*}
 \left|N[X]\right| &\ge& \sum_{x\in X}d_G(x)-(|X|-1)n\\
  &=&  e(X,B)-(s_{\ell-1}-1)n\\
  &\ge & n-\left(\left\lfloor k_{\ell}/2 \right\rfloor-1\right)(s_{\ell}+\left\lfloor k_{\ell}/2\rfloor-3 \right).
  \end{eqnarray*}
This proves Claim~2.\end{proof}

Next we consider the following four cases.

\medskip

{\bf Case 1.} $k_\ell=2$. By Claim 1, the subgraph of $G$ induced by the vertex set $X\cup (N(X))$ contains a copy of $F_{\ell-1}$.
Then $e(G-X)=0$, which implies that $e(G)=e(X,B)\leq s_{\ell-1}n=pn$, since $p=s_{\ell}-1=s_{\ell}-\lfloor k_{\ell}/2\rfloor=s_{\ell-1}=|X|$. On the other hand, by $k_{\ell}=2$ and (\ref{eq for 1}), $e(G)\ge s_{\ell-1}n=pn$. Hence $e(G)=s_{\ell-1}n$, which implies that $e(X,B)= pn$. Hence  $G=K_{s_{\ell}-1,n}\cup \overline{K}_{m-s_{\ell}+1}$ . 

\medskip

{\bf Case 2.} $k_\ell= 3$. Then there is at least one even number in $k_1,\ldots,k_{\ell-1}$.  By Claim 1, the subgraph of $G$ induced by the vertex set $X\cup N(X)$ contains a copy of $F_{\ell-1}$.
Note that $G$ is $F_\ell$-free. Hence it is easy to see that  each vertex of  $A\backslash X$ is not  adjacent to any neighbor of $X$ and $G-X$ consists of independent edges.
Thus \begin{eqnarray*}
       e(G) &\leq&s_{\ell-1}\left|N(X)\right|+n-\left|N(X)\right|\\
        &=& ( s_{\ell-1}-1)\left|N(X)\right|+n\\
        &\leq& ( s_{\ell-1}-1)n+n \\
        &=& s_{\ell-1}n=pn,
             \end{eqnarray*}
             since $p=s_{\ell}-1=s_{\ell}-\lfloor k_{\ell}/2\rfloor=s_{\ell-1}=|X|$.
On the other hand, by (\ref{eq for 1}), $e(G)\ge pn$. Hence $e(G)=pn$ which implies that  $\left|N(X)\right|=n$.  So $G=K_{p,n}\cup \overline{K}_{m-s_\ell+1}$.

\medskip

{\bf Case 3.}  $k_\ell=5$. By (\ref{eq for 1}) and $|X|=s_{\ell-1}$, we have
\begin{eqnarray*}
  e(G-X)&=& e(G)-e(X,B) \\
   &\geq&   (n-1)(s_\ell-1)+(m-s_\ell+1)-s_{\ell-1}n \\
   &=&n+m-2(s_\ell-1)>n.
\end{eqnarray*}
Hence there is a vertex $u\in A\setminus X$ with $d_G(u)\ge e(G-X)/m\ge n/m$, which implies that  $u$ has at least $n/m-\left(\left\lfloor k_{\ell}/2 \right\rfloor-1\right)(s_{\ell}+\left\lfloor k_{\ell}/2\rfloor-3 \right)$ neighbors in $N[X]$ by Claim~2.
Note that there is at least one even number greater than $5$, say $k_t$, in $k_1,\ldots,k_{\ell-1}$.
Then for every $v\in A\setminus (X\cup\{u\})$, $v$ has no neighbors in $N(X\cup\{u\})$ (otherwise, there exists a path of order $k_t$ in $G$ which contains $u$ and $v$. So $G$ contains a copy of $F_{\ell}$, a contradiction).
Denote by $|N(X\cup\{u\})|=r$.
 By (\ref{eq for 1}), $e(G)\geq (s_{\ell}-1)n+m-2s_{\ell}+2$.

 If $r=n$,  then by $s_{\ell}=s_{\ell-1}+\lfloor k_{\ell}/2\rfloor=s_{\ell-1}+2$,
\begin{eqnarray*}
 e(G) \leq |X\cup \{u\}|n=(s_{\ell-1}+1)n
   = (s_{\ell}-1)n<e(G),
\end{eqnarray*}
which is a contradiction.

If $r\leq n-2$, then  by $G-(X\cup\{u\})$ being $P_5$-free and  Theorem~\ref{THM: path in bipartite setting},
\begin{eqnarray*}
  e(G)&=& e(X\cup\{u\}, N(X\cup\{u\})+e(G-X\cup\{u\})\\
  &\leq& (s_{\ell-1}+1)r+\mbox{ex}(m-s_{\ell-1}-1,n-r;P_5) \\
   &=&  (s_{\ell-1}+1)r+(m-s_{\ell-1}-1+n-r)\\
   &=&s_{\ell-1}(r-1)+m+n-1 \\
   &\leq& s_{\ell-1}(n-3)+m+n-1\\
   &=&(s_{\ell}-1)n+m-3s_{\ell}+5< e(G),
\end{eqnarray*}
which is a contradiction.

Therefore $r=n-1$. Furthermore
 \begin{eqnarray*}
  e(G) \leq (s_{\ell-1}+1)(n-1)+m-s_{\ell-1}-1
   =  (s_{\ell}-1)n+m-2s_{\ell}+2\leq e(G).
\end{eqnarray*} Hence
$e(G)= (s_{\ell}-1)n+m-2s_{\ell}+2=p(n-\lfloor k_{\ell}/2\rfloor+1)+(m-p)(\lfloor k_{\ell}/2\rfloor-1)$  and $G=K_{p, n-\lfloor k_{\ell}/2\rfloor+1}\cup K_{m-p,\lfloor k_{\ell}/2\rfloor-1}$ by $s_{\ell-1}=s_{\ell}-\lfloor k_{\ell}/2\rfloor=s_{\ell}-2=p-1$.

\medskip

{\bf Case 4.}  Either $k_\ell=  4$ or $k_\ell\geq 6$.
If $e(X,B)\leq s_{\ell-1}\left(n-\left\lfloor  k_\ell/2 \right\rfloor+1\right)$, then
by (\ref{eq for 1}),  we have $e(A \setminus X, B)\geq (\lfloor  k_\ell/2 \rfloor-1)(n+m-s_{\ell-1}-2\lfloor  k_\ell/2 \rfloor+2)$.
In addition, since $G[A \setminus X , B]$ is $P_{k_\ell}$-free, by Theorem~\ref{THM: path in bipartite setting} and $m\geq 2p+1$,
$e(A \setminus X, B)\leq (\lfloor  k_\ell/2 \rfloor-1)(n+m-s_{\ell-1}-2\lfloor  k_\ell/2 \rfloor+2).$  Furthermore, it is easy to see that
$G=K_{s_\ell-1, n-\left\lfloor k_{\ell}/2 \right\rfloor+1}\cup K_{m-s_\ell+1,\left\lfloor k_\ell/2 \right\rfloor-1}$ and thus we are done.

If
\begin{equation}\label{E6}
e(X,B)> s_{\ell-1}\left(n-\left\lfloor  k_\ell/2 \right\rfloor+1\right),
\end{equation}
then
 there exists at least one vertex $u\in X$ with degree at least $n-\lfloor  k_\ell/2 \rfloor+2$.
Moreover, by (\ref{eq for 1}) and $e(X,B)\le |X|n=s_{\ell-1}n$, we have
\begin{equation}\label{eq 2}
 \begin{array}{lll}
  e(G-X)&=&e(G)-e(X,B)\\
 &   \geq& \left(n-\left\lfloor k_{\ell}/2\right\rfloor+1\right)(s_\ell-1)
+\left(m-s_\ell+1\right)\left(\left\lfloor k_{\ell}/2\right\rfloor-1\right) -s_{\ell-1}n\\
&=&\left(n+m-2s_{\ell}+2\right)(\lfloor k_{\ell}/2\rfloor-1).
  \end{array}
\end{equation}
If we delete vertices of degree at most $\lfloor k_{\ell}/2\rfloor-2$ in $G-X$ until that the resulting graph $G^\ast$ has no such vertex,
then the minimum degree of $G^\ast$ is at least $\lfloor k_{\ell}/2\rfloor-1$.
Let $Z$ be the set of  all deleted vertices of $G-X$ and $G^\ast=(A^\ast,B^\ast; E^\ast)=G-X-Z$.
Then we have 
$|Z|\leq(s_{\ell}+\lfloor k_{\ell}/2\rfloor-3)(\lfloor k_{\ell}/2\rfloor-1)$.
Otherwise, by (\ref{eq 2}) we have
\begin{equation}\label{eq 3}
 \begin{array}{lll}
e(G^\ast)&\geq&e(G-X)- (\lfloor k_{\ell}/2\rfloor-2) |Z|\\
&\geq&\left(n+m-2s_{\ell}+2\right)(\lfloor k_{\ell}/2\rfloor-1)- (\lfloor k_{\ell}/2\rfloor-2) |Z|\\
&=&\left(n+m-s_{\ell-1}- |Z|-1\right)(\lfloor k_{\ell}/2\rfloor-1)+|Z|\nonumber-(s_{\ell}+\lfloor k_{\ell}/2\rfloor-3)(\lfloor k_{\ell}/2\rfloor-1)\\
&> &(|A^\ast|+|B^\ast|-1) (\lfloor k_{\ell}/2\rfloor-1).
  \end{array}
\end{equation}
It follows from Corollary~\ref{Cor: path in bipartite setting} that $ G^\ast$ contains a copy of $P_{k_\ell}$, a contradiction (it is easy to find a copy of $F_{\ell-1}$  containing $X$).
Let $B^\prime$ be the common neighbours of $X$ in $B^\ast$. By Claim~2 and $|B^\ast| \geq n-(s_{\ell}+\lfloor k_{\ell}/2\rfloor-3)(\lfloor k_{\ell}/2\rfloor-1)$,
 we have $|B^\prime|\geq n-2(s_{\ell}+\lfloor k_{\ell}/2\rfloor-3)(\lfloor k_{\ell}/2\rfloor-1)$. Next we consider the following subcases.

 \medskip

{\bf Subcase 4.1.}   $k_\ell=4$.
Then  $d_G(u)=n$ and $G^\ast$ consists of stars.
Since $ |B^\ast|$ is sufficiently large with comparing to $|A^\ast|$,
there is a vertex $y \in A^\ast$ such that $y$ is adjacent is adjacent to at least $2s_{\ell}$ vertices of $B^\prime$.
If there is a vertex $z\in A\setminus  (X\cup \{y\})$ which is adjacent to any vertex of $B$,
then we can easily find a copy of $F_\ell$ in $G$ such that $P_4$ contains $u, z$ and $F_\ell-P_4$ contains $(X\setminus\{u\})\cup \{y\}$.
Hence we have $$e(G)\leq (s_{\ell-1}+1)n<(s_{\ell-1}+1)n+m-2s_{\ell}+2\leq e(G),$$ a contradiction. This finishes Subcase 4.1.

\medskip

{\bf Subcase 4.2.}  $k_\ell\geq 6$.
Hence, by Theorem~\ref{cycle}, $G^\ast$ contains  a cycle $C$ of length  $2\lfloor k_{\ell}/2\rfloor-2$ with partite sets $A_C\subseteq A^\ast$ and $B_C\subseteq B^\ast$.
Suppose that there is a vertex $z$ in $A\setminus (X\cup A_C)$ which is adjacent to at least $2s_\ell$ vertices of $B^\prime$.
If there is a vertex  $w\in N(u)$ which is adjacent to $v\in A\setminus (X\cup A_C\cup\{z\})$, then there is a copy of $F_{\ell}$ such that $F_{\ell-1}$ contains $(X\cup\{z\})\setminus\{u\}$  and $P_{k_\ell}$ contains $A_C\cup B_C \cup\{u,v,w\}$, a contradiction.
Thus each vertex of  $N(u)$ is not adjacent to any vertex of $A\setminus (X\cup A_C\cup\{z\})$.
Since each vertex in $B^\ast$ has degree $\lfloor k_\ell/2\rfloor-1$ in $G^\ast$,   each vertex of $N(u)\cap B^*$ is adjacent to at least one vertex of $A_C$.
Hence, it is easy to check that $G$ contains a copy of $F_\ell$ ($P_{k_{\ell}}$ contains $\{u\}\cup A_C\cup B_C$ and two vertices in $N(u)$), a contradiction.
Now we may suppose that each vertex in $A\setminus (X\cup A_C)$  is adjacent to at most $2s_\ell-1$ vertices of $B^\prime$. This implies that each vertex in $A\setminus (X\cup A_C)$  is adjacent to at most $2\lfloor k_{\ell}/2\rfloor(s_{\ell}+ \lfloor k_{\ell}/2\rfloor-2)$  vertices of $B$ and thus $e( A\setminus (X\cup A_C),B)\leq 2\lfloor k_{\ell}/2\rfloor(m-s_{\ell}+1)(s_{\ell}+ \lfloor k_{\ell}/2\rfloor-2)$.
Then by (\ref{eq 2}),
we have
\begin{eqnarray*}
  e(A_C,B) &=& e(G-X)-e( A\setminus (X\cup A_C),B)\\
   &\geq& \left(n+m-2s_{\ell}+2\right)(\lfloor k_{\ell}/2\rfloor-1)-2\lfloor k_{\ell}/2\rfloor(m-s_{\ell}+1)(s_{\ell}+ \lfloor k_{\ell}/2\rfloor-2)\\
   &=&(\lfloor k_{\ell}/2\rfloor-1)n-[(2\lfloor k_{\ell}/2\rfloor(s_{\ell}+ \lfloor k_{\ell}/2\rfloor-5/2)+1]m\\
   &&+2(s_{\ell}-1)[\lfloor k_{\ell}/2\rfloor(s_{\ell}+ \lfloor k_{\ell}/2\rfloor-3)+1]\\
   &\geq&(\lfloor k_{\ell}/2\rfloor-1)n-[(2\lfloor k_{\ell}/2\rfloor(s_{\ell}+ \lfloor k_{\ell}/2\rfloor-5/2)+1]m.
\end{eqnarray*}
Combining with $|A_C|= \lfloor k_{\ell}/2\rfloor-1$,  each vertex of $A_C$ is adjacent to at least $n-[(2\lfloor k_{\ell}/2\rfloor(s_{\ell}+ \lfloor k_{\ell}/2\rfloor-5/2)+1]m$ vertices of $B$ and hence the number of common neighbors of $X\cup A_C$ is at least $n-\epsilon n$, where $\epsilon$ is a small positive constant. Let $q_1$ be the number of vertices in $A\setminus (X\cup A_C)$ with degree at most one and $|N(X\cup A_C)|=q_2\geq n-\lfloor k_\ell/2\rfloor+2$.

Hence if all of $k_1,\ldots,k_\ell$ are odd, then each vertex of $A\setminus (X\cup A_C)$ with degree at least two is not adjacent to  any vertex in $N(X\cup A_C)$, as otherwise $G$ contains  a copy of $F_\ell$.
Then we have
 \begin{eqnarray*}
   e(G) &\leq&   q_1  + (s_\ell-1)q_2+\mbox{ex}(m-s_\ell-q_1+1,n-q_2;P_{k_\ell}) \\
    &=&  q_1  + (s_\ell-1)q_2+(m-s_\ell-q_1+1)(n-q_2) \\
    &=&  (m-s_\ell-q_1+1)n+(2s_\ell+q_1-m-2)q_2+q_1\\
   &\leq& (m-s_\ell-q_1+1)n+(2s_\ell+q_1-m-2)n+q_1\\
   &=&(s_\ell-1)n+q_1\\
    &\leq& (s_\ell-1)n+m-s_{\ell}+1\\
   &<&f(n,m,\lfloor k_\ell/2\rfloor,s_\ell),
 \end{eqnarray*}
 where the last strict inequality holds by $k_\ell\geq 9$ and $m\geq 2p+1$, a contradiction.
If not all $k_i$'s are odd, then no vertex of $A\setminus (X\cup A_C)$ is adjacent to $N(X\cup A_C)$.
Since $q_2\geq n-\lfloor k_\ell/2\rfloor+2$, we have
 \begin{eqnarray*}
   e(G) &\leq&   (s_\ell-1)q_2+\mbox{ex}(m-s_\ell+1,n-q_2;P_{k_\ell}) \\
    &=&  (s_\ell-1)q_2+(m-s_\ell+1)(n-q_2) \\
    &=&  (2s_\ell-2-m)q_2+(m-s_\ell+1)n\\
      &\leq&(2s_\ell-2-m)n+(m-s_\ell+1)n\\
      &=&(s_\ell-1)n\\
   &<&f(n,m,\lfloor k_\ell/2\rfloor,s_\ell),
 \end{eqnarray*}
  a contradiction.
 This finishes Subcase 4.2 and hence the proof of part (a) is complete.

\medskip

(b). All $k_i$'s are odd and $k_\ell\in \{3,5,7\}$.

\medskip

 Let $k_1=\cdots=k_\ell=3$ or  $k_1\geq \cdots \geq k_\ell=5$. Denote by  $G_1$  the graph obtained from taking a copy of $K_{s_\ell-1,n}$ and $m-s_\ell+1$ isolated vertices and joining the isolated vertices to  the partite set of $K_{s_\ell-1,n}$ with size $n$   by $m-s_\ell+1$ independent edges. Clearly, $G_1$ is $F_\ell$-free for $k_1=\cdots=k_\ell=3$  and $Z_{m,n}^p $ is $F_\ell$-free for  $k_1\geq \cdots \geq k_\ell=5$.
 Since $G$ is  an extremal bipartite $F_\ell$-free graph, we have
\begin{equation}
e(G)\geq(s_{\ell}-1)(n-1) +m.
\end{equation}
Similarly, as the proof of (a), there exists $X_1\subseteq A$ with size $s_{\ell-1} $ such that the number of common neighbors of $X_1$ is at least $c_1n$, where $c_1>0$.
If $k_1=\cdots=k_\ell=3$, then
 $G-X_1$ is $P_3$-free
and $e(G-X_1)\leq m-s_{\ell-1}$. Thus
$$e(G)=e(X_1,B)+e(G-X_1)\leq s_{\ell-1}n+m-s_{\ell-1}=(s_{\ell}-1)(n-1)+m\leq e(G),$$
which implies that $e(G)=(s_{\ell}-1)(n-1)+m$ and  $G=G_1$.
If $k_1\geq \cdots \geq k_\ell=5$.
Then $G-X_1$ is $P_5$-free, then
by Theorem~\ref{THM: path in bipartite setting},  we have $e(G-X_1)\leq n+m-s_{\ell-1}-1$.
Thus
$$e(G)=e(X_1,B)+e(G-X_1)\leq s_{\ell-1}n+n+m-s_{\ell-1}-1=(s_{\ell}-1)(n-1)+m\leq e(G).$$
which implies that $e(G)=(s_{\ell}-1)(n-1)+m$ and  $G=Z_{m,n}^p$.

Let $k_1\geq \cdots\geq k_{\ell-1}\geq  k_{\ell} =3$  with $k_1\geq 5$.
Clearly, $Z_{p+1,n}^p\cup \overline{K}_{m-p-1}$ is $F_\ell$-free. Since $G$ is  an extremal bipartite $F_\ell$-free graph, we have
\begin{equation}\label{eq for last case}
e(G)\geq(s_{\ell}-1)n+1.
\end{equation}
Similarly, as the proof of (a), there exists $X_2\subseteq A$ with size $s_{\ell-1} $ such that the number of common neighbors of $X_2$ is at least $c_2n$, where $c_2>0$.
Since $k_\ell=3$, $G-X_2$ consists of independent edges and isolated vertices.
By (\ref{eq for last case}), there is a vertex $x\in B$ with degree $s_{\ell-1}+1$, i.e., $x$ is adjacent to each vertex of $X_2$ and some vertex of $A\backslash X_2$.
Moreover, $x$ is the unique vertex in $B$ with at least degree $s_{\ell-1}+1$, as otherwise it is easy to see that $G$ contains a copy of $F_\ell$ ($k_1\geq 5$).
Hence, we have $e(G)\leq(s_{\ell}-1)n+1$.
Moreover, the equality holds if and only if $G=Z_{p+1,n}^p\cup \overline{K}_{m-p-1}$.

Let $k_\ell=7$ and $p+1\leq m\leq 3p$.
Since  $G$ is  an extremal bipartite $F_\ell$-free graph and $Z_{m,n}^p$ is  $F_\ell$-free,  we have $e(G)\geq p n+m-p$.
On the other hand, by Lemma~\ref{large common neighbours}, there is a set $X_3\subseteq A$ with size $p$ such that the number of common neighbors of $X_3$ is at least $c_3n$, where $c_3>0$.
Let $U$ be the set of $A \setminus X_3$ such that each vertex in $U$ has at least one neighbor in $N(X_3)$.
Then  it is easy to see that each vertex in $U$ have degree one and all vertices in $U$ share a unique common neighbor in $N(X_3)$, since   $G$  is  $F_\ell$-free and all $k_i$'s are odd.
Let $m_1=m-p-|U|$ and $n_1=n-\left|N(X_3)\right|$.
Hence, by $p+1 \leq m\leq 3p$, $p\geq 3$ and Lemma~\ref{lemma for P7 1}, we have
\begin{eqnarray*}
 e(G) &\leq& p(n-n_1) +(m-p-m_1)+\ex(m_1,n_1,P_{7})\\
  &\leq& pn +m-p+\ex(m_1,n_1,P_{7})-pn_1-m_1\\
   &\leq&pn+m-p.
\end{eqnarray*}
Hence $e(G)=pn+m-p$,  $|U|=m-p$ and $|N(X_2)|=n$,  which implies  $G=Z_{m,n}^p$.
So $\ex(m,n; F_{\ell})=pn+m-p$ for $p+1\leq m\leq 3p$. 
Let $m\geq 3p+1$.
Since $K_{p,n-2}\cup K_{2,m-p}$ is $F_\ell$-free, we have $e(G)\geq p(n-2)+2(m-p)$.
Note that Claim 2 still holds in this case.
We have $|N[X]|\geq n- 2s_{\ell}$.
Let $U$ be defined as before.
Then  by $m\geq 3p+1$, $p\geq 3$ and Lemma~\ref{lemma for P7 2}, we have
\begin{eqnarray*}
 e(G) &\leq& p(n-n_1) +(m-p-m_1)+\ex(m_1,n_1,P_{7})\\
   &\leq& p(n-2) +2(m-p)+\ex(m_1,n_1,P_{7})-p(n_1-2)-m+p- m_1\\
   &\leq&p(n-2)+2(m-p).
\end{eqnarray*}
Moreover, the equality holds if and only if $G=K_{p,n-2}\cup K_{2,m-p}$.
The proof of (b) is completed.
\QEDB

\section{Proofs of Theorem~\ref{linear forest} and Corollary~\ref{Cor}}

In order to prove  Theorem~\ref{linear forest}, we first prove the following key lemma.

\begin{Lemma}\label{Lem3.1}
 Let $F_{\ell}=\bigcup_{i=1}^{\ell} P_{k_i}$ be a linear forest with   ${\ell}\geq 1$ and
$k_1\geq \cdots \geq k_{\ell}\geq2$.
Let $ p=\sum _{i=1}^{\ell} \lfloor k_i/2 \rfloor-1\ge 1$
 If  $G$ is  an $F_{\ell}$-free connected  graph  of sufficiently large order $n$ with $\lambda(G)\geq \sqrt{p(n-p)}$, then  $K_{p,n-p}$ is a spanning subgraph of $G$.
\end{Lemma}

\begin{Proof}
Let $\lambda=\lambda(G)$ and $\mathbf x=(x_v)_{v\in V(G)}$ be a positive eigenvector to $\lambda $ such that  $w\in V(G)$ and $$x_w=\max\{x_v:v\in V(G)\}=1.$$
Set $L=\{v\in V(G): x_v> \epsilon\} $ and $S=\{v\in V(G): x_v\leq \epsilon\} $, where $c=2(p+ k_{\ell}-2)(4p+5)^{2p-1}(k_{\ell}-1)$ and $\epsilon$ to be chosen such that
$$\frac{\sqrt{2p+1}\sqrt{\sum_{i=1}^{\ell}k_i-2}}{\sqrt[4]{p(n-p)}}\leq\epsilon\leq\frac{1}{c}-\frac{p}{cn}.$$
Since $K_{p,n-p}$ is $F_{\ell}$-free,
\begin{eqnarray}\label{I2}
  \lambda &\geq& \lambda(K_{p,n-p})= \sqrt{p(n-p)}.
\end{eqnarray}
 By Theorem~\ref{Thm1.3} and Theorem 1.6 in \cite{Yuan2017},
\begin{eqnarray}\label{I3}
2e(S)&\leq& 2e(G)\leq (2p+1)n-p^2-2p\leq(2p+1)n.
\end{eqnarray}
We will finish our proof of the lemma in the following six claims.

\medskip

{\bf Claim~1.}     $2e(L)\leq  \epsilon n$ and  $e(L,S)<pn$.

\begin{proof}
By eigenequation of $A$ on any vertex $u\in L$, we have
$$\lambda\epsilon<\lambda x_u=\sum_{uv\in E(G)} x_v\leq d(u),$$
Thus $$2e(G)=\sum_{u\in V(G)}d(u)\geq \sum_{u\in L }d(u)\geq |L|\lambda\epsilon,$$
which implies that
\begin{eqnarray}\label{I0}
|L|\leq\frac{2e(G)}{\lambda\epsilon}\leq\frac{(2p+1)n}{\sqrt{p(n-p)}\epsilon}\leq \frac{\epsilon}{r-2} n,
\end{eqnarray}
where $r=\sum_{i=1}^{\ell} k_i$.   Since the subgraph induced by $L$ is $P_r$-free, it  follows from Theorem~\ref{Thm1.1} and (\ref{I0}) that
 $$2e(L)\leq( r-2)|L|\leq  \frac{(r-2)\epsilon n}{r-2}= \epsilon n.$$
Note that if $|L|\leq p$, then $e(L,S)< pn$;
 if $|L|\geq p+1$, then it follows from Theorem~\ref{linear forest in bipartite graph} that $e(L,S)< pn$. Hence in both situations, we have  $$e(L,S)< pn.$$
 This proves Claim~1.\end{proof}

{\bf Claim~2.}
For each $u\in L$, $ d(u)\geq (1-(3p+3)(1-x_u+\epsilon))n$.

\begin{proof}
Denote $B_u=\{v\in V(G): uv\notin E(G)\}$.
Then
\begin{eqnarray*}
  \lambda \sum_{v\in V(G)}x_v &=& 
  \sum_{v\in V(G)} \sum_{vz\in E(G) } x_z\\
  &=& 
   \sum_{v\in L}  d(v)x_v + \sum_{v\in S} d(v)x_v \\
  &\leq& \sum_{v\in L}  d(v)+\epsilon \sum_{v\in S} d(v)\\ 
   &=&  2e(L)+2\epsilon e(S)+(1+\epsilon)e(L,S),
\end{eqnarray*}
which implies that
\begin{equation}\label{4}
 \sum_{v\in V(G)}x_v\leq \frac{ 2e(L)+2\epsilon e(S)+(1+\epsilon)e(L,S)}{\lambda}.
\end{equation}
By  (\ref{4}) and Lemma~\ref{spec1}, we have
\begin{equation*}\label{5}
  \begin{aligned}
 \frac{1}{  \lambda }|B_u|&\leq \sum_{v\in B_u}x_v\leq\sum_{v\in V(G)}x_v-\sum_{uv\in E(G)} x_v=\sum_{v\in V(G)}x_v-\lambda x_u\\
   &\leq  \frac{ 2e(L)+2\epsilon e(S)+(1+\epsilon)e(L,S)}{\lambda}-\lambda x_u,
     \end{aligned}
\end{equation*}
which implies that
\begin{eqnarray*}
  |B_u| &\leq&  2e(L)+2\epsilon e(S)+(1+\epsilon)e(L,S)-\lambda^2x_u \\
   &<&  \epsilon n+\epsilon(2p+1)n+(1+\epsilon)pn-p(n-p)x_u \\
 &=&((3p+2)\epsilon+(1-x_u)p)n+p^2x_u\\
   &\leq& (3p+5/2)(1-x_u+ \epsilon)n,
\end{eqnarray*}
where the last second inequality holds as $p^2x_u\leq p^2\leq\frac{\epsilon n}{2}\leq\frac{(1-x_u+\epsilon)n}{2}$.
Hence $$d(u)\geq n-1-(3p+5/2)(1-x_u+ \epsilon)n\geq(1-(3p+3)(1-x_u+\epsilon))n.$$
This proves Claim~2.\end{proof}

{\bf Claim~3.}  For a given integer $s$ with $ 1 \leq s<p $, if   there exists a set $X$ of $s$ vertices such that $X=\{v\in L: x_v\geq 1-\eta ~\text{and} ~d(v)\geq(1 - \eta)n\} $, where $\eta$ is much smaller than $1$, then there exists  a vertex $v\in L \backslash X $ such that  $ x_{v}\geq 1 - 12(p+1)^2(\eta + \epsilon)$ and  $d(v)\geq (1 - 12(p+1)^2(\eta + \epsilon))n $.
\vspace{2mm}

\begin{proof}
Denote $t=|L\cap X|$.
By eigenequations of $A^2$ on vertex  $w$, we have
\begin{eqnarray*}
 \lambda^2&=&\lambda^2x_w=\sum\limits_{vw\in E(G)}\sum\limits_{uv\in E(G)}x_u\\
 &\leq& \sum\limits_{v\in V(G)}\sum\limits_{uv\in E(G)}x_u=\sum\limits_{uv\in E(G)}(x_u+x_v)\\
   &=& \sum\limits_{uv\in E(S)}(x_u+x_v) +\sum\limits_{uv\in E(L)}(x_u+x_v)+\sum\limits_{uv\in E(L,S)}(x_u+x_v)\\
   &\leq&  2\epsilon e(S)+2e(L)+\sum\limits_{uv\in E(L,S)}(x_u+x_v)\\
  & \leq& 2\epsilon e(S)+2e(L)+\epsilon e(L,S)+\sum_{\substack{ uv\in E( L\backslash X,S)\\u \in L\backslash X }}x_u+\sum_{\substack{uv\in E(L\cap X,S)\\u\in L\cap X}}x_u,
\end{eqnarray*}
which implies that
\begin{eqnarray*}
 \sum_{\substack{uv\in E(L\backslash X,S)\\ u\in L\backslash X}}x_u
 &\geq&  \lambda^2- 2\epsilon e(S)-2e(L)-\epsilon e(L,S)-\sum_{\substack{uv\in E( L\cap X,S)\\u\in L\cap X}}x_u\\
       &\geq&p(n-p)-\epsilon(2p+1)n-\epsilon n-\epsilon pn-tn\\
       &=&(p-t-\epsilon (3p+2))n-p^2\\
       &\geq&(p-t-\epsilon (3p+3))n
    \end{eqnarray*}
    where the last inequality holds as $\epsilon \geq\frac{ p^2}{n} $.
In addition, since
$$e(L\cap X, S)+e(L\cap X,L\backslash X)+2e(L\cap X)=\sum_{v\in L\cap X} d(v)\geq t(1-\eta)n,$$
we have
\begin{eqnarray*}
  e(L\cap X, S)  &\geq& t(1-\eta)n-e(L\cap X,L\backslash X)-2e(L\cap X) \\
   &\geq& t(1-\eta)n-t(|L|-t)-t(t-1) \\
   &\geq&  t(1-\eta)n-t(\epsilon n-t)-t(t-1)\\
      &=& t(1-\eta-\epsilon )n+t.
\end{eqnarray*}Hence we have
  $$e(L\backslash X,S)= e(L,S)-e(L\cap X, S)\leq  pn-t(1-\eta-\epsilon )n-t <(p-t(1-\eta-\epsilon ))n.$$
Let $$f(t)=\frac{p-t-\epsilon (3p+3)}{p-t(1-\eta-\epsilon )}.$$  It is easy to see that $f(t)$ is decreasing with respect to $1\leq t\leq s \leq p-1$.
Then \begin{eqnarray*}
  \frac{\sum\limits_{\substack{uv\in E(L\backslash X,S)\\ u\in L\backslash X}}x_u}{e(L\backslash X,S)} &\geq& f(t)\geq f(p-1)
     =\frac{1-\epsilon (3p+3)}{1+(p-1)(\eta+\epsilon)}\geq1-(4p+2)(\eta+\epsilon).
\end{eqnarray*}
Hence there exists a vertex $u\in L\backslash X$ such that
 $$ x_u\geq1-(4p+2)(\eta+\epsilon)>1-12(p+1)^2(\eta+\varepsilon).$$
By Claim~2,
\begin{eqnarray*}
d(u) &\geq& (1-(3p+3)((4p+2)(\eta+\epsilon)+\epsilon))n\\
 &\geq &(1-(3p+3)(4p+3)(\eta+\epsilon))n\geq(1-12(p+1)^2(\eta+\epsilon))n.
\end{eqnarray*}
This proves Claim~3.\end{proof}

{\bf Claim~4.} There exists a set $A$ of size $p$ such that $x_v\geq 1-(4p+5)^{2p-1}\epsilon$ and $d(v)\geq(1-(4p+5)^{2p-1}\epsilon)n$ for any $v\in A$. Moreover,  $|N[A]|\geq(1-p(4p+5)^{2p-1}\epsilon)n$.

\begin{proof}
Let $v_1=w $. Then $x_{v_1}=1$. By Claim 2,
$$d(v_1)\geq (1-(3p+3)\epsilon)n.$$
 Let $\eta_1=(3p+3)\epsilon$. Applying Claim~3 iteratively,  we can get another distinct $p-1$ vertices  $v_2,\ldots, v_{p}$ of $L$  such that for $i=2,\ldots,p$,
$$x_{v_i}\geq 1-\eta_i$$ and $$d(v_i)\geq (1-\eta_i)n,$$  where $\eta_i=12(p+1)^2(\eta_{i-1}+\epsilon)$.
Let $A=\{v_1,v_2,\ldots,v_p\}$.  By the definition of $\eta_1,\ldots, \eta_{p}$, we have
$$\eta_1\leq\dots\leq\eta_{p}=\left(3\times12^{p-1}(p+1)^{2p-1}+\sum_{i=1}^{p-1}12^i(p+1)^{2i}\right)\epsilon\le(4p+5)^{2p-1}\epsilon.$$
Then
$$\left|\bigcap_{i=1}^{p} N(v_i)\right|\geq \sum_{i=1}^{p}(1-\eta_i)n-(p-1)n=\left(1-\sum_{i=1}^{p}\eta_i\right)n\geq(1-p(4p+5)^{2p-1}\epsilon)n.$$
The proof of Claim 4 is complete.\end{proof}

\medskip

{\bf Claim~5.} Denote   $B=N[A]$ and $R=V(G)\backslash (A\cup B)$. Then
 $x_v\leq\frac{1}{k_{\ell}-1}$ for  $v\in B\cup R$.
\begin{proof}
We first show that $|N_{G}(v)\cap(A\cup B)|\leq p+2$  for any $v\in  B\cup R$.
In fact, for any vertex $v\in R$, $|N_{G}(v)\cap B|\leq 1$. Otherwise, we can embed a copy of $F_{\ell}$ in $G$ as $|B|\geq p+\ell+1$, which is  a contradiction. Furthermore, by the definition of $R$, $|N_{G}(v)\cap A|\leq p-1$, which implies that
\begin{equation}\label{claim5-1}|N_{G}(v)\cap(A\cup B)|=|N_{G}(v)\cap B|+|N_{G}(v)\cap A|\leq 1+p-1=p.\end{equation}
For any vertex $v\in  B$, $|N_{G}(v)\cap B|\leq2$, otherwise we can embed an $F_{\ell}$ in $G$ as $|B|\geq p+\ell+1$, which is  a contradiction.  So  for any $v\in  B$,
\begin{equation}\label{claim5-3}
|N_{G}(v)\cap(A\cup B)|=|N_{G}(v)\cap B|+|N_{G}(v)\cap A|\leq 2+p=p+2.\end{equation}
Hence  by (\ref{claim5-1}) and (\ref{claim5-3}),  we have $|N_{G}(v)\cap(A\cup B)|\leq p+2$  for any $v\in  B\cup R$.

Note that $G[R]$ is  $P_{k_{\ell}}$-free, as otherwise we can embed  a copy of  $F_{\ell}$ in $G$ as $|B|\geq p+\ell+1$, which is  a contradiction. By Theorem~\ref{Thm1.1},
$2e(R)\leq (k_{\ell}-2)|R|$.
By (\ref{claim5-1}),
\begin{eqnarray*}
  \sum_{v\in R} x_v &=&  \frac{1}{\lambda }\sum_{v\in R} \sum_{uv\in E(G)} x_u
   \leq \frac{1}{\lambda } \sum_{v\in R}d(v)=\frac{1}{\lambda }( 2e(R)+e(R,A\cup B))\\
   &\leq&  \frac{1}{\lambda }(2e(R)+p|R|)\leq \frac{(k_{\ell}-2)|R|+p|R|}{\lambda}=\frac{(p+k_{\ell}-2)|R|}{\lambda}.
\end{eqnarray*}
Hence for any $v\in  B\cup R$,
\begin{eqnarray*}
  x_v &=& \frac{1}{\lambda } \sum\limits_{uv\in E(G)} x_u\leq \frac{1}{\lambda }\sum_{\substack{uv\in E(G)\\u\in A\cup B}}x_u+\frac{1}{\lambda }\sum_{\substack{uv\in E(G)\\u\in R}}x_u  \leq \frac{p+2}{\lambda }+\frac{1}{\lambda }\sum_{u\in R}x_u\\
   &\leq& \frac{p+2}{\lambda}+  \frac{(p+k_{\ell}-2)|R|}{\lambda^2}\\
   &\leq&\frac{p+2}{\sqrt{p(n-p)}} +\frac{(p+k_{\ell}-2)p(4p+5)^{2p-1}\epsilon n}{p(n-p)}\\
   &\leq&\frac{1}{2(k_{\ell}-1)}+\frac{1}{2(k_{\ell}-1)}
   =\frac{1}{k_{\ell}-1},
\end{eqnarray*}
So Claim~5 holds.\end{proof}


%

{\bf Claim~6.}  $R=V(G)\backslash (A\cup B)$ is empty.

\begin{proof}
Assume for contradiction that  $R$ is not empty. Note that $G[B\cup R]$ is $P_{k_{\ell}}$-free, otherwise  we can embed  a copy of  $F_{\ell}$ in $G$ as $|B|\geq p+\ell+1$. By Theorem~\ref{Thm1.1}, there is a vertex $v\in R$ with at most $k_{\ell}-2$ neighbors in $R$.
Let $H$ be a graph obtained from $G$ by removing all edges incident with $v$ and then connecting $v$ to each vertex in $A$. Clearly,
 $H$ is still $F_{\ell}$-free. By the definition of $R$, $v$ can be adjacent to at most $p-1$ vertices in $A$.
Let $u\in A$ be a vertex which is not adjacent to $v$ (if there are at least two such vertices, choose $u$ arbitrarily.  By Claims~4 and 5 again,
\begin{eqnarray*}
  \lambda(H)- \lambda &\geq& \frac{{\bf x}^T A(H){\bf x}}{ {\bf x}^T{\bf x}}-\frac{{\bf x}^T A{\bf x}}{ {\bf x}^T{\bf x}}
 \geq\frac{2x_v}{{\bf x}^T{\bf x}}\left( x_u-\sum_{\substack{vz\in E\\ z\in B\cup R}} x_z\right)\\
   &\geq&  \frac{2x_v}{{\bf x}^T{\bf x}}\left(1-(4p+5)^{2p-1}\epsilon-\frac{k_{\ell}-2}{k_{\ell}-1}\right)\\
   &=& \frac{2x_v}{{\bf x}^T{\bf x}}\left( \frac{1}{k_{\ell}-1}-(4p+5)^{2p-1}\epsilon  \right)\\
   &>&0,
\end{eqnarray*}
which implies that $\lambda(H)>\lambda$, a contradiction. This proves Claim~6.\end{proof}

By Claim~6, we have $V(G)=A\cup B$, where $|A|=p$ and $|B|=n-p$. So $K_{p,n-p}$ is spanning subgraph of $G$.
\end{Proof}

\medskip

\noindent{\bf Proof of Theorem~\ref{linear forest}.}  Let $G$ be an $F_{\ell}$-free bipartite  graph of   order  $n$ with the maximum spectral radius. Since $K_{p,n-p}$ is $F_{\ell}$-free,
\begin{eqnarray}\label{I2}
  \lambda(G) &\geq& \lambda(K_{p,n-p})=\sqrt{p(n-p)}.
\end{eqnarray}
If   $G$ is connected, then by Lemma~\ref{Lem3.1}, $K_{p,n-p}$ is spanning subgraph of $G$. Since $G$ is a bipartite graph, we have $G=K_{p,n-p}$.
Assume that $G$ is not connected. Let $G_1$ be a component of $G$  such that $\lambda(G_1)=\lambda(G)$.  Set $n_1=|V(G_1)|$.
Note that $G$ is triangle-free. By Wilf theorem \cite[Theorem~2]{Wilf1986}, we have
\begin{eqnarray*}
\frac{  n_1^2}{4} &\geq&  \lambda^2(G_1)=\lambda^2(G)\geq p(n-p),
\end{eqnarray*}
which implies that  $n_1\geq 2\sqrt{p(n-p)}$, i.e., $n_1$ is also sufficiently large. By Case~1,
$$\lambda(G)=\lambda(G_1)\leq \sqrt{p(n_1-p)}<\sqrt{p(n-p)},$$ which is a contradiction.
This completes the proof.
\QEDB
\vspace{3mm}

\noindent {\bf Proof of Corollary~\ref{Cor}.} By a result of Favaron, Mah\'{e}o and Sacl\'{e} \cite{FMS}, we have $\lambda_n(G) \geq \lambda_n(H)$  for some spanning bipartite subgraph $H$. Moreover, the equality holds if and only if $G= H$, which can be deduced by its original proof.
 By Theorem~\ref{linear forest},
$$\lambda(H)\leq \sqrt{p(n-p)}$$ with equality if and only if $H = K_{p,n-p}$.
Since  the spectrum of a bipartite graph is symmetric \cite{LP},  $$\lambda_n(H)\geq-\sqrt{p(n-p)}$$ with equality if and only if $H= K_{p,n-p}$.  Thus
 we have $$ \lambda_n(G)\geq -\sqrt{p(n-p)}$$
with equality if and only if $G= K_{p,n-p}$.
\QEDB


\begin{thebibliography}{12}

\bibitem{alon2003}N.~Alon,  M.~Krivelevich, B.~Sudakov,  Tur\'{a}n numbers of bipartite graphs and related Ramsey-type questions. Special issue on Ramsey theory. {\it Combin. Probab. Comput.}  {\bf 12} (2003), no. 5-6, 477--494.

\bibitem{babai2009}L.~Babai, B.~Guiduli,  Spectral extrema for graphs: the Zarankiewicz problem. {\it  Electron. J. Combin.} 16 (2009), no. 1, Research Paper 123, 8 pp.



%



	%
	
 \bibitem{Bushwa2011}   N. Bushaw, N. Kettle, Tur\'{a}n numbers of multiple paths and equibipartite forests, {\it Combin. Probab. Comput. } {\bf 20} (2011),  837--853.

\bibitem{campos2015} V.~Campos, R.~Lopes, A proof for a conjecture of Gorgol, {\it Electron. Notes  Discrete Math.}  {\bf 50} (2015), 367--372.
\bibitem{Chen2019} M.-Z.~Chen, A.-M.~Liu, X.-D.~Zhang, Spectral extremal results with forbidding linear forests, {\it Graphs Combin.} {\bf 35} (2019), 335--351.	
\bibitem{Erdos1959} P. Erd\H{o}s, T. Gallai, On maximal paths and circuits of graphs, {\it Acta Mathematica Hungarica} {\bf 10(3)} (1959), 337--356.
	
\bibitem{erdos1966}P.~Erd\H{o}s, M.~Simonovits,  A limit theorem in graph theory. {\it Studia Sci. Math. Hungar. } {\bf 1} (1966), 51--57.
\bibitem{erdos1946}P.~Erd\H{o}s, A.~H.~Stone, On the structure of linear graphs. {\it Bull. Amer. Math. Soc.} {\bf} 52 (1946), 1087--1091.
\bibitem{FMS} O. Favaron, M. Mah\'{e}o, J.-F. Sacl\'{e}, Some eigenvalue properties in graphs (conjectures of Graffiti-II), {\it Discrete Math.} {\bf 111} (1993), 197--220.
\bibitem{furedi1996-1}Z.~F\"{u}redi, New asymptotics for bipartite Tur\'{a}n numbers. {\it J. Combin. Theory Ser. A}  {\bf 75} (1996), no. 1, 141--144.
\bibitem{furedi1996-2}Z.~F\"{u}redi,  An upper bound on Zarankiewicz' problem. {\it Combin. Probab. Comput.} {\bf 5 }(1996), no. 1, 29--33.

\bibitem{furedi2013}Z.~F\"{u}redi, M.~Simonovits,  The history of degenerate (bipartite) extremal graph problems. Erd\H{o}s centennial, 169--264, Bolyai Soc. Math. Stud., 25, J\'{a}nos Bolyai Math. Soc., Budapest, 2013.
	\bibitem{gorgol} I.~Gorgol, Tur\'{a}n numbers for disjoint copies of graphs, {\it Graphs Combin.} {\bf 27} (2011), 661-667.
	\bibitem{Gyarfas-Rousseau-Schelp-1984} A. Gy\'arf\'as, C. C. Rousseau, R. H. Schelp, An extremal problem for paths in bipartite graphs. {\it J. Graph Theory}  {\bf 8(1)} (1984), 83--95.
\bibitem{jackson1981}{ B.~Jackson}, { Cycles in bipartite graphs}{\it J. Combin. Theory Ser. B}{ \bf 30}{1981}{ 332--342}

\bibitem{jiang2020}T.~Jiang, Y.~Qiu,  Tur\'{a}n numbers of bipartite subdivisions. {\it SIAM J. Discrete Math.} {\bf 34} (2020), no. 1, 556--570.
	\bibitem{Kopylov1977} G.~N.~Kopylov, On maximal paths and cycles in a graph, {\it Soviet Math. Dokl,} {\bf 18} (1977), 593--596.
	\bibitem{lidicky2013} B.~Lidick\'{y}, H.~Liu, C.~Palmer, On the Tur\'{a}n number of forests, {\it Electron. J. Combin.}   {\bf 20(2)}  (2013), Paper 62, 13 pp.
\bibitem{LP}    L.~Lov\'{a}sz, J.~Pelik\'{a}n, On the eigenvalues of trees, {\it Period. Math. Hung. } {\bf 3 } (1973), 175--182.



 \bibitem{Nikiforov2010}   V.~Nikiforov, The spectral radius of graphs without paths and cycles of specified length, {\it  Linear Algebra Appl.} {\bf 432} (2010), 2243--2256.
      \bibitem{Nikiforov2011}V.~Nikiforov,  Some new results in extremal graph theory. Surveys in combinatorics 2011, 141-181, London Math. Soc. Lecture Note Ser., 392, Cambridge Univ. Press, Cambridge, 2011.

     \bibitem{nikiforov2018}V.~Nikiforov, M.~Tait, C.~Timmons,  Degenerate Tur\'{a}n problems for hereditary properties. {\it Electron. J. Combin.}  {\bf 25} (2018), no. 4, Paper No. 4.39, 11 pp.
\bibitem{ning2020}B.~Ning, J.~Wang, The formula for Tur\'{a}n number of spanning linear forests. {\it Discrete Math.}  {\bf 343} (2020), no. 8, 111924, 6 pp.

\bibitem{sudakov2020}B.~Sudakov, I.~Tomon,  Tur\'{a}n number of bipartite graphs with no $K_{t,t}$,  {\it Proc. Amer. Math. Soc.} 148 (2020), no. 7, 2811--2818.
 \bibitem{Wilf1986}   H.S.~Wilf, Spectral bounds for the clique and independence numbers of graphs. {\it J. Combin. Theory Ser. B}  {\bf 40 } (1986), 113--117.
	\bibitem{Yuan2017} L.T.~~Yuan, X.-D.~Zhang, The Tur\'{a}n number of  disjoint copies of paths, {\it  Discrete Math.} {\bf 340(2)} (2017), 132--139.
	
	\bibitem{Yuan2019.2} L.-T.~Yuan, X.-D.~Zhang, Tur\'{a}n numbers for disjoint paths,  {\it  J. Graph Theory}, {\bf 98(3)} (2021), 499--524.
\bibitem{Yuan2019.1} L.-T.~Yuan, X.-D.~Zhang, Extremal graphs for linear forest in bipartite graphs,  {\it Discuss. Math. Graph Theory}, to appear.
\bibitem{Zhai2015} M.-Q.~Zhai, H.-Q~Lin, S.-C.~Gong, Spectral conditions for the existence of specified paths and cycles in graphs, {\it Linear Algebra Appl.} {\bf 471} (2015) 21--27.

    \bibitem{zhai2021}M.~Zhai, H.~Lin, J.~Shu,  Spectral extrema of graphs with fixed size: cycles and complete bipartite graphs. {\it European J. Combin.}  {\bf 95} (2021), Paper No. 103322, 18 pp.







%

\end{thebibliography}
\end{document}